\documentclass[12pt]{article}
\usepackage{graphicx} 
\usepackage[a4paper, total={6.5in, 9in}]{geometry} 
\usepackage{amsfonts}
\usepackage{amsthm}
\usepackage{amsmath}
\usepackage[sortcites = true, backend=biber, style=numeric, url = false]{biblatex} 
\usepackage[shortlabels]{enumitem} 
\usepackage{hyperref}
\hypersetup{
    colorlinks=true,
    linkcolor=blue,
    filecolor=magenta,      
    urlcolor=cyan,
    pdftitle={Overleaf Example},
    pdfpagemode=FullScreen,
    }
\usepackage[capitalize]{cleveref} 
\usepackage{xcolor} 
\usepackage{quiver} 
\usepackage{orcidlink}
\usepackage{authblk} 
\usepackage{physics} 

\usepackage{commands}

\newenvironment{keywords}
{
\begin{center}
\textbf{Keywords}\\
\vspace{0.17cm}
\begin{minipage}{14.5cm}}
{\footnotesize
\end{minipage}
\end{center}}

\newenvironment{Abstract}
{
\begin{center}
\textbf{Abstract}\\
\vspace{0.25cm}
\begin{minipage}{14.5cm}}
{\footnotesize
\end{minipage}
\end{center}}

\newcommand{\ran}{\operatorname{rank}} 
\newcommand{\codim}{\operatorname{codim}} 
\newcommand{\mail}[1]{\small\href{mailto:#1}{#1}}

\newtheorem{theorem}{Theorem}[subsection]
\newtheorem{proposition}{Proposition}[subsection]
\newtheorem{lemma}{Lemma}[subsection]

\theoremstyle{definition}
\newtheorem{example}{Example}[subsection]
\newtheorem{Def}{Definition}[subsection]
\newtheorem{remark}{Remark}[subsection]
\newtheorem{obs}{Observation}[subsection]

\begin{document}
\title{\huge Coisotropic reduction in Multisymplectic Geometry}

\author{Manuel de León \orcidlink{0000-0002-8028-2348} \\ \mail{mdeleon@icmat.es}}
\affil{Instituto de Ciencias Matemáticas, Campus Cantoblanco, Consejo Superior de Investigaciones Científicas, C/Nicolás Cabrera, 13–15, Madrid 28049, Spain}
\affil{Real Academia de Ciencias Exactas, Físicas y Naturales de España, C/Valverde, 22, Madrid 28004, Spain}

\author{Rubén Izquierdo-López \orcidlink{0009-0007-8747-344X} \\ \mail{rubizqui@ucm.es}}
\affil{Instituto de Ciencias Matemáticas, Campus Cantoblanco, Consejo Superior de Investigaciones Científicas, C/Nicolás Cabrera, 13–15, Madrid 28049, Spain}

\date{\today}

\maketitle

\begin{Abstract}
In this paper we study coisotropic reduction in multisymplectic geometry. On the one hand, we give an interpretation of Hamiltonian multivector fields as Lagrangian submanifolds and prove that $k$-coisotropic submanifolds induce a Lie subalgebra in the algebra of Hamiltonian $(k-1)$-forms, similar to how coisotropic submanifolds in symplectic geometry induce a Lie subalgebra under the Poisson bracket. On the other hand, we extend the classical result of symplectic geometry of projection of Lagrangian submanifolds in coisotropic reduction to bundles of forms, which naturally carry a multisymplectic structure.
\end{Abstract}


\tableofcontents

\section{Introduction}

Multisymplectic geometry is the natural framework in which to formulate classical field theories, just as symplectic geometry plays that key role in Lagrangian and Hamiltonian mechanics \cite{abraham2008foundations,ClassicalFieldsSniaiycki,dLR}. Indeed, the bundles of exterior forms are naturally equipped with a multisymplectic form, in the same way that for the bundle of $1$-forms (i.e. the cotangent bundle of the manifold) the natural structure is a symplectic form. However, multisymplectic geometry exhibits a much higher degree of complexity, dealing with differential forms of higher degree.  These differences make multisymplectic geometry richer but at the same time more complicated, and if the holy grail of classical field theories is to seek a full extension of the results in symplectic mechanics, this task is far from being fully achieved. This paper tries to cover some aspects that have already been partially dealt with in previous papers \cite{Ibort1999OnTG,HamiltonianStructuresIbort}, thus initiating an ambitious plan that we hope to complete in the coming years.\\

One of the key aspects of this new approach is not to consider any notion of regularity in the definition of a multisymplectic form, as is usually done in applications to classical field theories \cite{KdVequationGotay,GotayMultisymplecticFramework,MarkSpaceTimeDecomposition,Ibort1999OnTG,Roman_Roy_2009,NarcisoMultisymplecticFormalism,Invitation2019}. This allows us to work with greater flexibility, recovering regularity as a particular case. Our main objective in this paper is to study the submanifolds of a multisymplectic manifold, in particular the relations between Lagrangian and coisotropic submanifolds \cite{Ibort1999OnTG,deleon2003tulczyjews,Sevestre_2019}. In doing so, we prove a coisotropic reduction theorem which generalises the one already known for symplectic geometry. The interest of this reduction lies in the fact that the Lagrangian submanifolds are the geometric interpretation of the dynamics, and if one of them has a clean intersection with a coisotropic one, it can be reduced to the quotient of the latter while maintaining the Lagrangian character (and so, providing a reduced dynamics) \cite{WeinsteinLagrangianNeighborhood,abraham2008foundations}. Very relevant by-products of these notions and results are the construction of graded brackets and the interpretation of a coisotropic submanifold in terms of these brackets, as well as the study of currents and conserved quantities \cite{HamiltonianStructuresIbort,FORGER_2003,Blacker_2021} (see also \cite{AitorSymmetries,ConservedQuantitiesMarco}).
We would like to mention that the graded brackets that are used in this paper are related to the notion of higher-Poisson structures (see \cite{Bursztyn2015}), a generalization of the notion of a Poisson structure. \\

The paper is structured as follows. \cref{section:MultisymplecticGeometry} introduces the fundamental concepts of both the multilinear version of symplectic geometry in the realm of vector spaces and the corresponding translation to the realm of differentiable manifolds. In this section we introduce the main examples of multisymplectic vector spaces and multisymplectic manifolds. In the first case, we pay special atention to multisymplectic structures of forms arising from a vector space together with a ``vertical" subspace. Similarly, in the second case we study bundles of forms over a manifold together with a regular ``vertical" distribution. In \cref{Section:Hamiltonianstructures}, we develop the notion of Hamiltonian vector fields and Hamiltonian forms; it should be noticed that we do not ask for any regularity conditions from the multisimplectic forms, so we have to work with the respective kernels to avoid singularities. Thus we can interpret multivector fields as Lagrangian submanifolds of multisymplectic manifolds by naturally extending the results known in symplectic geometry. At the same time, we complete the results of previous work, which allow us to introduce a graded Lie algebra of brackets. We can also consider an abstract framework for the study of currents and conserved quantities. Finally, in \cref{section:CoisotropicSubmanifolds} we obtain the extension of the coisotropic reduction theorem as well as the reduction of Lagrangian submanifolds via coisotropic reduction. To do that, we need to extend some theorems on coisotropic manifolds due to Weinstein. The paper ends with some conclusions and a list of potential future work in \cref{conclusions}.

\section{Multisymplectic geometry}\label{section:MultisymplecticGeometry}
The main concepts and results of multisymplectic vector spaces and manifolds are taken from \cite{Ibort1999OnTG,HamiltonianStructuresIbort,deleon2003tulczyjews}. As mentioned, we chose a more general approach which includes possibly degenerate forms. Most of the proofs are included for the sake of completeness.
\subsection{Multisymplectic vector spaces}\label{Subsection:MultisymplecticVectorSpapces}
\begin{Def}[Multisymplectic vector space] A \textbf{multisymplectic vector space} of order $k$ is a pair $(V, \omega),$ where $\omega$ is a $(k+1)$-form on $V$, namely $\omega \in \bigwedge^{k+1} V^\ast.$ The multisymplectic vector space and the form will be called \textbf{non-singular} or \textbf{regular} if the map given by contraction $$V \xrightarrow{\flat_1} \bigwedge^k V; \,\, v \mapsto \iota_v \omega$$ defines a monomorphism, that is, $\iota_v \omega = 0$ only when $v = 0.$
\end{Def}
\begin{obs} This terminology is not standard. In the literature, an arbitrary form $\omega \in \bigwedge^k V^\ast$ is usually called \textit{pre-multisymplectic}, but we choose this terminology for the sake of simplicity. We prefer this general approach because in \cref{Section:Hamiltonianstructures}, ``singular'' multisymplectic manifolds (what we simply call multisymplectic) appear naturally. Nevertheless, all the definitions given in the text coincide with the usual definitions when $\omega$ is non-degenerate. 
\end{obs}
The $(k +1)$-form $\omega$ induces the following map:
\begin{Def} Let $(V, \omega)$ be a multisymplectic vector space of order $k$. Define the map induced by contraction $$\flat_q : \bigwedge^{q} V \rightarrow \bigwedge^{k + 1 - q} V^\ast, \,\, u \mapsto \iota_u \omega.$$
\end{Def}
\begin{remark} A multisymplectic vector space of order $k$, $(V, \omega)$ is regular if and only if $\ker \flat_1 = 0.$
\end{remark}

\begin{obs} For $k = 1$, all possible forms are classified up to linear isomorphism. Indeed, it is a well known fact that $\omega_1, \omega_2 \in \bigwedge^2 V^\ast$ are in the same $GL(V)$-orbit if and only if $\rank \omega_1 = \rank \omega_2.$ In particular, when $\dim V$ is even, every pair of non-degenerate $2$-forms are in the same $GL(V)$-orbit. For general order $k > 2$, the classification is far from trivial. For example, $3$-forms are not classified for arbitrary $\dim V$. For a complete table of the number of $GL(V)$-orbits in $\bigwedge^k V^\ast$, we refer to \cite{orbitsforms}. 
\end{obs}
The isomorphism of multisymplectic vector spaces is given by the following definition.
\begin{Def}[Multisymplectomorphism] Let $(V_1, \omega_1),$ $(V_2, \omega_2)$ be multisymplectic vector spaces. A \textbf{multisymplectomorphism} between $(V_1, \omega_1)$ and $(V_2, \omega_2)$ is a linear isomorphism $$f: V_1 \rightarrow V_2$$ satisfying $$f^\ast \omega_2 = \omega_1.$$ 
\end{Def}

\begin{example}\label{ex:canonicalmultisymplectic}Let $L$ be a vector space and take $V := L \oplus \bigwedge^{k+1} L$ with $k \leq \dim V$. Define the $(k +1)$-form $$\Omega_L ((v_1, \alpha_1), \dots, (v_{k+1}, \alpha_{k+1})) := \sum_{j = 1}^{k+1} \alpha_j(v_1, \dots, \hat{v}_j, \dots, v_{k+1}),$$ where $\hat{v}_j$ means that the $j$th-vector is missing. Then, $\Omega_L$ is a regular multisymplectic form and, thus, $(V, \Omega_L)$ is a regular multisymplectic vector space. 
\end{example}

Similar to the notion of orthogonal in symplectic vector spaces, we can define a (now indexed) version in multisymplectic vector spaces.
\begin{Def}[Multisymplectic orthogonal] Let $(V, \omega)$ be a multisymplectic vector space of order $k$, $W \subseteq V$ be a subspace and $1 \leq j \leq k$. Define the ${j}$\textbf{th-orthogonal} to $W$ as the subspace $$W^{\perp, j} = \{v \in V : \,\, \iota_{v \wedge w_1 \wedge \cdots \wedge w_j} \omega = 0, \forall w_1, \dots, w_j \in W\}.$$
\end{Def}
It can be easily proved that the $j$th-orthogonal satisfies the following properties:
\begin{proposition}\label{prop: propertiesoforthogonal} Let $(V, \omega)$ be a multisymplectic vector space of order $k$. Then,
\begin{enumerate}[a)]
    \item $\{0\}^{\perp, j} = V$ for all $1 \leq j \leq k$;
    \item $V^{\perp, 1} = \ker \flat_1$;
    \item $(W_1 + W_2)^{\perp, j} \subseteq W_1^{\perp, j} \cap W_2^{\perp, j},$ for all $1 \leq j \leq k$, and for all subspaces $W_1, W_2 \subseteq V$;
    \item $W_1^{\perp, j} + W_2^{\perp, j} \subseteq (W_1 \cap W_2)^{\perp, j}$ for all $1 \leq j \leq k$, and for all subspaces $W_1, W_2 \subseteq V$;
    \item $(W_1 + W_2)^{\perp, 1} \subseteq W_1^{\perp, 1} \cap W_2^{\perp, 1},$ for all subspaces $W_1, W_2 \subseteq V$.
\end{enumerate}
\end{proposition}
The definitions of isotropic, coisotropic, Lagrangian and symplectic generalize as follows:
\begin{Def}[$j$-isotropic, $j$-coisotropic, $j$-Lagrangian, multisymplectic] Let $(V, \omega)$ be a multisymplectic vector space of order $k$. A subspace $W \subseteq V$ will be called
\begin{enumerate}[a)]
    \item $j$\textbf{-isotropic}, if $W \subseteq W^{\perp, j}$;
    \item $j$\textbf{-coisotropic}, if $W^{\perp, j} \subseteq W + \ker \flat_1$;
    \item $j$\textbf{-Lagrangian}, if $W = W^{\perp, j} + \ker \flat_1$;
    \item \textbf{non-degenerate}, if $W \cap W^{\perp, 1} = 0.$
\end{enumerate}
\end{Def}
\begin{obs} Notice that when $\omega$ is regular, $\ker \flat_1 = 0$, and we recover the standard definitions of $j$-isotropic, $j$-coisotropic, and $j$-Lagrangian.
\end{obs}
\begin{proposition}\label{prop:characterizationisotropic} Let $(V, \omega)$ be a multisymplectic vector space of order $k$. Then, a subspace $i: W \rightarrow V$ ($i$ being the natural inclusion) is $k$-isotropic if and only if $$i^\ast \omega = 0.$$
\end{proposition}
\begin{proof}$W$ is $k$-isotropic if and only if $$\omega(w_1, \dots, w_{k+1}) = 0,$$ for every $w_1, \dots, w_{k+1} \in W$ or, equivalently, $i^\ast \omega = 0.$
\end{proof}
 \begin{example}\label{ex:verticalforms} Let $L$ be a vector space and $\mathcal{E} \subseteq L$ be a proper subspace. For $r \geq 0$, $k \leq \dim L$, define $$\bigwedge^k_r L^\ast := \left\{\alpha \in \bigwedge^k L^\ast : \,\, \iota_{v_1 \wedge \cdots \wedge v_r} \alpha = 0,\, \forall v_1, \dots, v_r \in \mathcal{E}\right\}.$$Notice that, if $r \leq \dim \mathcal{E}$, for the subspace $\bigwedge^k_r L^\ast$ to be non trivial, we need to ask $k - r + 1 \leq \codim \mathcal{E}.$ Then, under these conditions and for $r \geq 2$, $$L \oplus \bigwedge^k_r L^\ast$$ is a non-degenerate subspace of $(L \oplus \bigwedge^k L^\ast, \Omega_L)$ from \cref{ex:canonicalmultisymplectic} and, consequently, $$\left(L \oplus \bigwedge^k_r L^\ast, i^\ast\Omega_L\right)$$ is a regular multisymplectic vector space, where $i$ is the natural inclusion. 
\end{example}

From now on, we will denote $\Omega_L$ as the multisymplectic form in $L \oplus \bigwedge ^k_r L^\ast,$ making abuse of notation.

\begin{obs} Notice that for $r > \dim \mathcal{E}$, or $\mathcal{E} = 0$, we recover the canonical multisymplectic vector space $L \oplus \bigwedge^k L$. For simplicity, we will refer to this case as $r = 0$. The only degenerate case is for $r = 1$ and we have $$\ker \flat_1 = \mathcal{E}.$$
\end{obs}

\begin{remark}\label{hypotheseskr}For the sake of clarity in the exposition, we will assume throughout the rest of this section the hypotheses that make $(L\oplus \bigwedge^k_r L^\ast, \Omega_L)$ a regular multisymplectic vector space. More precisely, we will assume $k \leq \dim L$ and, when $r \neq 0$,
\begin{itemize}
    \item $k - r +1 \leq \codim \mathcal{E}$;
    \item $1 < r \leq \dim \mathcal{E}$.
\end{itemize}
Any further hypotheses will be made explicit in the corresponding results.
\end{remark}

\begin{proposition}[\cite{Ibort1999OnTG}] Identify  both $L$ and $W := \bigwedge^k_r L^\ast$ a subspace of $L \oplus \bigwedge^k_r L ^\ast$. Then $L$ is $k$-Lagrangian, and $W$ $1$-Lagrangian in $$\left( L \oplus \bigwedge^k_r L^\ast, \Omega_L\right).$$ 
\end{proposition}
\begin{proof}  It is clear that $L$ is $k$-isotropic and that $W$ is $1$-isotropic. To see that $L$ is $k$-coisotropic, let $(l, \alpha) \in L \oplus \bigwedge^k_r L^\ast$ such that $$\Omega_L((l, \alpha), (l_1, 0), \dots, (l_k, 0)) = 0,$$ for every $l_1, \dots, l_k \in L,$ that is, $$\alpha(l_1, \dots, l_k) = 0,$$ for every $l_1, \dots, l_k \in L.$ We conclude $\alpha = 0$, and thus, $$L^{\perp, k} = L = L + \ker \flat_1 \footnote{If $r \neq 1, \ker \flat_1 = 0$ and, if $r = 1$, $\ker \flat_1 = \mathcal{E}$. In any case, the equality $L = L + \ker \flat_1$ holds.}$$

\noindent Now, to see that $W$ is $1$-Lagrangian, let $(l, \alpha) \in L \oplus \bigwedge^k_r L^\ast$ such that $$\Omega_L((l, \alpha), (0, \beta_1), (l_2, \beta_2), \dots, (l_k, \beta_k)) = 0,$$ for every $\beta_1, \dots, \beta_k \in \bigwedge^k_r L^\ast$, and $l_2, \dots, l_k \in L.$ Then, $$\beta_1(l, l_2, \dots, l_k) = 0,$$ for every $l_2, \dots, l_k \in L.$ Now we distinguish two cases:
\begin{enumerate}
    \item\underline{Case $r \neq 1$}. Then, necessarily $l = 0$, concluding $$W^{\perp, 1} = W = W + \ker \flat_1,$$ because $\ker \flat_1 = 0.$
    \item \underline{Case r = 1}. If $$\beta_1(l, l_2, \dots, l_k) = 0,$$ for every $l_2, \dots, l_k \in L$, we have $l \in \mathcal{E}$ and, therefore, $$(l, \alpha) \in W + \ker \flat_1,$$ proving that $W$ is $1$-Lagrangian.
\end{enumerate}
\end{proof} 

An important class of multisymplectic vector spaces are those that are multisymplectomorphic to those of \cref{ex:canonicalmultisymplectic} and \cref{ex:verticalforms}. First observe the following:
\begin{proposition}\label{prop: characterization} A non-degenerate multisymplectic vector space $(V, \omega)$ is multisymplectomorphic to the one defined in \cref{ex:verticalforms} if and only if there exists $\mathcal{E}, L, W \subseteq V$ satisfying:
\begin{itemize}
    \item $L$ is $k$-Lagrangian with $\mathcal{E} \subseteq L$;
    \item $W$ is $1$-Lagrangian and, if $e_1, \dots, e_r \in \mathcal{E}$, we have $$\iota_{e_1 \wedge \cdots \wedge e_r} \omega = 0;$$
    \item $V = L \oplus W$ and $$\dim W = \dim \bigwedge ^k_r L^\ast,$$ where the vertical forms are taken with respect to $\mathcal{E}.$
\end{itemize}
\end{proposition}
\begin{proof} It is clear the hypothesis imply that the following linear map $$\phi: W \rightarrow \bigwedge^k_r L^\ast; \,\, \alpha \mapsto (\iota_\alpha \omega) |_L$$ defines a linear isomorphism. Now, let $\Phi$ be the isomorphism given by $$\Phi:= id_L \oplus \phi: V = L \oplus W \rightarrow L \oplus \bigwedge^k_r L^\ast.$$ We have $\Phi^\ast \Omega_L = \omega.$ Indeed,
\begin{align*}
    &(\Phi^\ast \Omega_L) (l_1 + \alpha_1, \dots, l_{k +1}+ \alpha_{k +1}) = \Omega_L((l_1, \phi(\alpha_1), \dots, (l_{k+1}, \phi(\alpha_{k+1}))) = \\
    &= \sum_{j = 1}^{k+1} (-1)^{j+1} (\phi(\alpha_j))(l_1, \dots, \hat{l}_j, \dots, l_{k+1})= \sum_{j= 1}^{k+1} (-1)^{j+1}\omega( \alpha_j, l_1, \dots, \hat{l}_j, \dots, l_{k+1}) \\
    &= \sum_{j= 1}^{k+1} \omega( l_1, \dots, l_{j-1}, \alpha_j, l_{j+1} \dots, l_{k+1}) = \omega(l_1 + \alpha_1, \dots, l_{k +1}+ \alpha_{k +1}),
\end{align*}
proving the result.
\end{proof}
We can prove a weaker version of \cref{prop: characterization}. Indeed, given $\mathcal{E}$, $L$, $W$ satisfying the hypotheses, we can canonically identify $\mathcal{E}$ as a subspace of $V/W$ via the isomorphism $$V/W \cong L.$$ It is easily verified that $$\iota_{ e_1 \wedge \cdots \wedge e_r} \omega = 0,$$ for all $e_1, \dots, e_r \in \mathcal{E}$ is equivalent to $$\iota_{v_1 \cdots \wedge v_r} \omega = 0,$$ for all $v_1, \dots, v_r \in V$satisfying $\pi(v_i) \in \mathcal{E}$ (identifying $\mathcal{E}$ as a subspace of $V/W$), for every $1 \leq i \leq r.$ We have the following:
\begin{theorem}[\cite{deleon2003tulczyjews}]\label{thm:type(kr)} A non-degenerate multisymplectic vector space $(V, \omega)$ is multisymplectomorphic to $(L \oplus \bigwedge^k_r L^\ast, \Omega_L)$ if and only if there exists $ W \subseteq V$ and $\mathcal{E} \subseteq V/W$ satisfying:
\begin{itemize}
    \item $W$ is $1$-Lagrangian and, for all $v_1, \dots, v_r \in V$ with $\pi(v_i) \in \mathcal{E}$ (with $\pi: V \rightarrow V/W$ the canonical projection), we have $$\iota_{ v_1 \wedge \cdots \wedge v_r} \omega = 0;$$
    \item There is an equality of dimensions $$\dim W = \dim \bigwedge ^k_r V/W,$$ where the vertical forms are taken with respect to $\mathcal{E}.$
\end{itemize}
\end{theorem}
To prove it, we will need the following Proposition, which will also be useful in the sequel:

\begin{proposition}\label{isotropicdecomposition} Let $(V, \omega)$ be a multisymplectic vector space, and $U$, $W$ be $k$-isotropic, and $1$-isotropic subspaces respectively such that
$$V = U \oplus W.$$
Then, $U$ is $k$-Lagrangian, and $W$ is $1$-Lagrangian.
\end{proposition}
\begin{proof} We need to prove that $$U^{\perp, k} = U + \ker \flat_1.$$ Let $u + w \in U^{\perp, k},$ for $u \in U$, $w \in W$. Then, for all $u_1, \dots, u_k \in U$ we have $$\omega(u + w, u_1, \dots, u_k) = \omega(w, u_1, \dots, u_k) = 0,$$ where we have used that $U$ is $k$-isotropic. We claim that $w \in \ker \flat_1.$ Indeed, given $v_i \in V$, for $i = 1, \dots, k$, we can write $v_i = u_i + w_i$, with $u_i \in U$, $w_i \in W$. Then, $$\omega(w, v_1, \dots, v_k) = \omega(w, u_1 + w_1, \dots, u_k +w_k) = \omega(w, u_1, \dots, u_k) = 0,$$ where in the last equality we used that $W$ is $1$-isotropic. Therefore, if $u + w \in U^{\perp, k},$ we have $$u + w \in U +\ker \flat_1,$$ that is $$U^{\perp, k} \subseteq U + \ker \flat_1,$$ proving that $U$ is $k$-coisotropic and, therefore, $k$-Lagrangian.\\

\noindent To show that $W$ is $1$-Lagrangian, let $u + w \in W^{1, \perp}$, with $u \in U$, $w \in W$. Then $u \in \ker \flat_1.$ Let $v_i = u_i + + w_i$, $u_i \in U$, $w_i \in W$, $1 \leq i \leq k$. Since $W$ is $1$-isotropic, for every $1 \leq i \leq k$ $$\iota_{u \wedge w_i} \omega = \iota_{(u + w) \wedge w_i} \omega = 0.$$ Now, using that $U$ is $k$-isotropic, and $W$ is $1$-isotropic,
$$\omega(u, u_1 + w_1, \dots, u_k + w_k) = \sum_{j = 1}^{k} \omega(u, u_1,\dots, u_{j-1}, w_j,u_{j+1}, \dots, u_k) = 0.$$
Therefore, $u \in \ker \flat_1$ and $u + w \in W + \ker \flat_1,$ showing that $$W^{\perp, 1} = W + \ker \flat_1,$$ ending the proof.
\end{proof}

\begin{proof}[Proof (of \cref{thm:type(kr)})] The proof we give mimics the case $r = 0$ from \cite{Sevestre_2019}. It is enough to show the existence of a $k$-Lagrangian complement to $W$. By \cref{isotropicdecomposition}, we will conclude the proof once we show that there exists a $k$-isotropic complement.\\

\noindent First observe that, since $W$ is $1$-Lagrangian, $\iota_\alpha \omega$ induces a form on $V/W$ defining $$\phi(\alpha)(\pi(v_1), \dots, \pi(v_{k})) := \omega(\alpha, v_1, \dots, v_k),$$ for every $\alpha \in W$. This map defines a linear isomorphism $$\phi: W \rightarrow \bigwedge^k_r (V/W)^\ast,$$ where the vertical forms are taken with respect to $\mathcal{E}.$ Take $L$ any complement to $W$ in $V$ and define the linear isomorphism $$\Phi:= id_L \oplus \phi: V = L \oplus W \rightarrow L \oplus \bigwedge^k_r (V/W)^\ast.$$ We will look for subspaces of the form $\Phi^{-1} \circ \mathbf{A} (L),$ where $\mathbf{A} = id_L \oplus A,$ with $$A: L\rightarrow \bigwedge^k_r(V/W)^\ast.$$
For this subspace to be $k$-isotropic, it has to satisfy $$\omega(\Phi^{-1} \circ \mathbf{A} (l_1), \dots, \Phi^{-1} \circ \mathbf{A}(l_{k+1})) = 0,$$ for all $l_1, \dots, l_{k+1} \in L.$ We have 
\begin{align*}
    \omega(\Phi^{-1} \circ \mathbf{A} (l_1), \dots, \Phi^{-1} \circ \mathbf{A}(l_{k+1})) = \omega(l_1 + \Phi^{-1} A ( l_1), \dots, l_{k+1} + \Phi^{-1} A ( l_{k+1})) = \\
    \omega(l_1, \dots, l_{k+1}) + \sum_{j= 1}^{k+1}(-1)^{j+1} \omega(\Phi^{-1}A(l_j), l_1, \dots, \hat{l}_j, \dots, l_{k+1})\\
    =  \omega(l_1, \dots, l_{k+1}) + \sum_{j= 1}^{k+1}(-1)^{j+1} (A(l_j)) (\pi(l_1), \dots, \hat{l}_j, \dots,  \pi(l_{k+1})).
\end{align*}
Notice that the projection $\pi$ restricted to $L$ defines an isomorphism $$\pi|_L: L \rightarrow V/W.$$
Define $A$ closing the following diagram\footnote{Notice that these functions are well defined. Indeed, $\iota_l \omega \in \bigwedge^k_r L ^\ast$ for any $l \in L$.}
\[\begin{tikzcd}
	L && {\bigwedge^k_rL^\ast} \\
	&& {\bigwedge^k_r (V/W)^\ast}
	\arrow["\psi", from=1-1, to=1-3]
	\arrow["A"', curve={height=12pt}, from=1-1, to=2-3]
	\arrow["{\pi^\ast}"', from=2-3, to=1-3]
\end{tikzcd},\] where $$\psi(l) := -\frac{\iota_l \omega}{k+1}.$$Then,
\begin{align*}
    &\sum_{j= 1}^{k+1}(-1)^{j+1} (A(l_j)) (\pi(l_1), \dots, \hat{l}_j, \dots,  \pi(l_{k+1}) = \sum_{j= 1}^{k+1}  (-1)^{j+1}(\pi^\ast A(l_j))(l_1, \dots, \hat{l}_j, \dots, l_{k+1})= \\
    &\sum_{j= 1}^{k+1} \frac{(-1)^{j}}{ k+1} \omega(l_j,l_1, \dots, \hat{l}_j, \dots, l_{k+1}) = - \omega(l_1, \dots, l_{k+1}),
\end{align*}
concluding $$\omega(\Phi^{-1} \circ \mathbf{A} (l_1), \dots, \Phi^{-1} \circ \mathbf{A}(l_{k+1})) = 0,$$ and proving the result.
\end{proof}

This induces the following definition:
\begin{Def}[Multisymplectic vector space of type $(k,r)$] A \textbf{multisymplectic vector space of type $(k,r)$} is a tuple $(V,\omega, W, \mathcal{E})$ satisfying the hypothesis of \cref{thm:type(kr)}.
\end{Def}

In later considerations, the next lemma will be very useful:

\begin{lemma}\label{lemma1} Let $(V, \omega, W, \mathcal{E})$ be a multisymplectic vector space of type $(k,r)$. Then, denoting by $\flat_1$ the induced map $$V \xrightarrow{\flat_1} \bigwedge^k V^\ast,$$ we have $$\bigwedge^k_{1,r} V^\ast \subseteq \flat_{1}( V),$$ where $$\bigwedge^k_{1, r} V^\ast  = \bigwedge ^k_1 V^\ast \cap \bigwedge^k_r V^\ast, $$ and the vertical forms are taken with respect to $W$ and $\mathcal{E}$\footnote{This latter meaning that $\iota_{e_1 \wedge \cdots \wedge e_r} \alpha = 0$, for every $e_1, \dots,e_r \in V$ with $\pi(e_i) \in \mathcal{E}$, where $\pi: V \rightarrow V/W$ is the canonical projection.}, respectively.
\end{lemma}
\begin{proof} By \cref{thm:type(kr)}, it is enough to prove it in the canonical case $V = L \oplus \bigwedge_r^k L^\ast, W = \bigwedge^k_r L^\ast.$ Then, any $k$-form $\alpha \in \bigwedge^k_{1,r} V$ is the pull-back of a $k$-form $\widetilde \alpha \in \bigwedge^k_r L^\ast$. An elementary calculation proves that $$\iota_{\widetilde \alpha} \Omega_L = \alpha.$$
\end{proof}

\subsection{Multisymplectic manifolds}
For an introduction to the  study of general (non-degenerate) multisymplectic manifolds motivating their study from Classical Field Theory see \cite{Ibort1999OnTG,Invitation2019}. Also, for an in-depth treatment of the multisymplectic formulation of Classical Field Theory, see \cite{ClassicalFieldsSniaiycki,silvia}.
\begin{Def}[Multisymplectic manifold] A \textbf{multisymplectic manifold} of order $k$ is a pair $(M, \omega)$, where $M$ is a manifold, and $\omega$ is a \textbf{multisymplectic form} of order $k$, that is, a closed $(k +1)$-form. $(M, \omega)$ will be called \textbf{non-degenerate} if $\omega_x$ is non-degenerate for every $x \in M$.
\end{Def}

The $j$th-orthogonal complement defined in \cref{Subsection:MultisymplecticVectorSpapces} and the notion of $j$-isotropic, $j$-coisotropic, $j$-Lagrangian and regular subspaces generalizes to distributions $\Delta$, and submanifolds $N$, defining it in each subspace $\Delta_x$, or in each tangent space $T_x N.$

\begin{example} \label{ex:formsmanifold}We can generalize \cref{ex:canonicalmultisymplectic} to manifolds. Fix a manifold $L$ and define $$M:= \bigwedge^k L,$$ the bundle of $k$-forms. We have the tautological $k$-form $$\Theta^k_L|_{\alpha_x}(v_1, \dots, v_k):= \alpha_x(\pi_\ast v_1, \dots, \pi_\ast v_k),$$ where $\pi: \bigwedge^k L \rightarrow L$ is the natural projection. Define $$\Omega^k_L := - d\Theta^k_L.$$ Then $(\bigwedge^k L, \Omega^k_L)$ is a non-degenerate multisymplectic manifold of order $k$. It is easy to check (see \cref{lemma:completelift}) that $\Theta^k_L$ and $\Omega^k_L$ are the only forms on $\bigwedge^k L$ satisfying $$\alpha^\ast \Theta^k_L = \alpha, \,\, \alpha^\ast \Omega^k_L = - d \alpha,$$ for every $k$-form (interpreted as a section) $\alpha: L \rightarrow \bigwedge^k L$.
\end{example}

 In canonical coordinates $(x^i, p_{i_1, \dots, i_k})$ on $\bigwedge^k L$, we have $$\Theta^k_L = p_{i_1, \dots, i_k} d x^{i_1} \wedge \cdots \wedge dx^{i_k},$$ and $$\Omega^k_L = -dp_{i_1, \dots, i_k} \wedge d x^{i_1} \wedge \cdots \wedge dx^{i_k}.$$ This immediately shows that the vertical distribution $W_L^k$ associated to the vector bundle $\bigwedge^k L \rightarrow L$ is $1$-isotropic. Additionally, we have:

 \begin{proposition}\label{prop:KLagrangianclosed} $W_L^k$ defines a $1$-Lagrangian distribution. Furthermore, a form (interpreted as a section) $$\alpha: L \rightarrow \bigwedge^k L$$ defines a $k$-Lagrangian submanifold if and only if it is closed.
 \end{proposition}
 \begin{proof} By \cref{isotropicdecomposition}, it is enough to show that $\alpha$ defining a $k$-isotropic submanifold is equivalent to $\alpha$ being closed (this would imply that $\alpha(L)$ is $k$-Lagrangian, and that $W^k_L$ is $1$-Lagrangian, since they are complementary). Indeed, by \cref{prop:characterizationisotropic}, $\alpha(L)$ is $k$-isotropic if and only if $$ 0 = \alpha^\ast \Omega^k_L = - d \alpha, $$ that is, if and only if $\alpha$ is closed.
 \end{proof}

 There is another relevant type of multisymplectic manifolds that generalizes \cref{ex:verticalforms}:
 \begin{example}\label{ex:verticalformmanifold} Let $L$ be a manifold and $\mathcal{E}$ be a regular distribution, where $r, k, \mathcal{E}_x, T_x L$ are in the hypotheses of \cref{hypotheseskr} for every $x \in L$. Define $$\bigwedge^k_r L:= \left\{ \alpha_x \in \bigwedge^k T^\ast_xL: \,\, \iota_{e_1 \wedge \cdots \wedge e_r} \alpha_x = 0,\,\, \forall e_1, \dots, e_r \in \mathcal{E}_x\right\}.$$ It is easy to check that $\bigwedge^k_r L$ defines a non-singular submanifold of $\bigwedge^k L$. Therefore, $$\left(\bigwedge^k_r L, \Omega_L\right)$$ is a multisymplectic manifold of order $k$.
 \end{example}

 \begin{remark}Just like in \cref{Subsection:MultisymplecticVectorSpapces}, throughout the rest of the text we will assume the conditions that make $\bigwedge^k_r L$ a regular multisymplectic manifold.
 \end{remark}

 A natural question to ask is what are the necessary (and sufficient) conditions for a multisymplectic manifold $(M, \omega)$ to be locally multisymplectomorphic to either of the models presented in \cref{ex:formsmanifold} or in \cref{ex:verticalformmanifold}. Of course, if it were the case, the multisymplectic vector space $(T_x M, \omega_x)$ would necessarily be of type $(k, r)$ (for the corresponging values in the model).
 
 \begin{Def}[\cite{deleon2003tulczyjews} Multisymplectic manifold of type $(k,r)$] A multisymplectic manifold of type $(k,r)$ is a tuple $(M, \omega, W, \mathcal{E})$, where $(T_xM,\omega_x, W_x, \mathcal{E}_x)$ is a multisymplectic vector space of type $(k,r)$ and $W$ is a regular integrable distribution.
 \end{Def}
 In \cite{Martin1988ADT}, G. Martin gave the characterization for multisymplectic manifolds of type $(k, 0)$.
   \begin{theorem}[\cite{Martin1988ADT} Darboux theorem for multisymplectic manifolds of type $(k,0)$] Let $(M, \omega, W)$ be a multisymplectic manifold of type $(k,0)$. Then, around each point $x \in M$ there exists a neighborhood $U$ of $x$ in $M$, a manifold $L$, and a multisymplectomorphism $$\phi: (U, \omega) \rightarrow (V, \Omega_L)$$ where $V$ is an open subset of $\bigwedge^k L $. 
 \end{theorem}
 And, in \cite{deleon2003tulczyjews}, M. de León et. al. generalized the result to multisymplectic manifolds of type $(k, r)$.
   \begin{theorem}[\cite{deleon2003tulczyjews} Darboux theorem for multisymplectic manifolds of type $(k,r)$] \label{thm:darbouxmultisymplectic} Let $(M, \omega, W, \mathcal{E})$ be a multisymplectic manifold of type $(k,r)$. Then, around each point $x \in M$ there exists a neighborhood $U$ of $x$ in $M$, a manifold $L$, and a multisymplectomorphism $$\phi: (U, \omega) \rightarrow (V, \Omega_L)$$ where $V$ is an open subset of $\bigwedge^k_r L $.
   \end{theorem}

   For a recent review on ``Darboux type Theorems'' in geometric structures appearing in the geometric formulation of Classical Field Theories, we refer to \cite{gràcia2023darboux}.

\section{Hamiltonian structures on multisymplectic manifolds}\label{Section:Hamiltonianstructures}
\subsection{Hamiltonian multivector fields and forms. Brackets}
\begin{Def}[\cite{HamiltonianStructuresIbort} Hamiltonian multivector field, Hamiltonian form] Let $(M, \omega)$ be a multisymplectic manifold of order $k$. A multivector field $$U: M \rightarrow \bigvee_q M$$ will be called a \textbf{Hamiltonian multivector field} if there exists a $(k-q)$-form on $M$, $\alpha$, such that $$\iota_U \omega = d \alpha.$$ In this context, $\alpha$ is called the \textbf{Hamiltonian form} associated to $U$. Furthermore, $U$ will be called a \textbf{locally Hamiltonian multivector field} if $\iota_U\omega$ is closed. Of course, if $U$ is Hamiltonian, it is locally Hamiltonian.
\end{Def}

We will denote by $\mathfrak{X}^q_H(M)$ the space of all Hamiltonian multivector fields of order $q$, and by $\Omega^{l}_H(M)$ the space of all Hamiltonian $l$-forms.\\

There is certain ``correspondence" between Hamiltonian multivector fields and Hamiltonian forms. However, this correspondance is not well defined, a Hamiltonian multivector field $U$ can be associated to different Hamiltonian forms, and viceversa. Nevertheless, if $$\iota_U\omega = d \alpha = d \beta,$$ for some $\alpha, \beta \in \Omega^l_H(\Omega)$, we have that $$d(\alpha - \beta) = 0.$$ Therefore, we obtain a well defined epimorphism $$\mathfrak{X}^q_H(M) \xrightarrow{\flat_q} \Omega^{k - q}_H (M) /Z^{k-q}(M) =: \Hform{k-q}{M},$$ where $Z^{k-q}(M)$ is the space of all closed forms, mapping each Hamiltonian multivector field $U$ to the class of Hamiltonian forms $[\alpha]$ satisfying $$\iota_U\omega = d \alpha.$$ We would like this map to be inyective, and we can achieve this by quotienting $\mathfrak{X}^q_H(M)$ by $\ker \flat_q$, which is the space of all multivector fields $U$ satisfying $$\iota_U \omega = 0.$$ Therefore, defining $$\Hmultivector{q}{M} := \mathfrak{X}_H^q(M) /\ker \flat_q,$$ we obtain isomorphisms between the spaces
\[\begin{tikzcd}
	{\Hmultivector{1}{M}} & \cdots & {\Hmultivector{q}{M}} & \cdots & {\Hmultivector{k}{M}} \\
	\\
	{\Hform{k-1}{M}} & \cdots & {\Hform{k-q}{M}} & \cdots & {\Hform{0}{M}}
	\arrow["{\flat_1}", tail reversed, from=1-1, to=3-1]
	\arrow["{\flat_q}", tail reversed, from=1-3, to=3-3]
	\arrow["{\flat_k}", tail reversed, from=1-5, to=3-5]
\end{tikzcd}.\]
Of course, these isomorphisms induce an isomorphism between the corresponding graded vector spaces 
$$\SHmultivector{M} := \bigoplus_{q = 1}^{k} \Hmultivector{q}{M} \xrightarrow{ \flat} \SHform{M} := \bigoplus_{q = 1}^{k} \Hform{k - q}{M}.$$ We can try to endow these spaces with a graded Lie algebra structure. Given the isomorphism, it would be enough to define the bracket in one of the spaces and obtain the induced bracket in the other via the $\flat$ mappings. 
\begin{proposition}[\cite{HamiltonianStructuresIbort}]\label{bracketmultivectorfields} Let $(M, \omega)$ be a multisymplectic manifold, and $U , V$ be Hamiltonian multivector fields of degree $p, q$, respectively. Then, $[U, V]$ is a Hamiltonian multivector field of degree $p + q - 1$, where $[\cdot, \cdot]$ denotes the Schouten-Nijenhuis bracket (see \cite{vaisman2012lectures}).
\end{proposition}
\begin{proof} We have the equality (see \cite{vaisman2012lectures}) $$\iota_{[U, V]} \omega =  - d \iota_{U \wedge V} \omega,$$ which proves the proposition.
\end{proof}
Given the equality $$\iota_{[U, V]}\omega = - d \iota_{U \wedge V} \omega$$ from \cref{bracketmultivectorfields}, we have that whenever $U \in \ker \flat_p$ (or $V \in \ker \flat_q$),  then 
$$
[U, V] \in \ker \flat_{p +q - 1}.
$$ 
Therefore, we obtain a well defined bracket $$\Hmultivector{p}{M} \times \Hmultivector{q}{M} \rightarrow \Hmultivector{p +q -1}{M}; \,\, (\widehat{U}, \widehat{V}) \mapsto [\widehat{U}, \widehat{V}]:= \widehat{[U, V]},$$ where $\widehat{U}$ denotes the class of $U$ modulo $\ker \flat_q.$ By the previous considerations, we define the induced bracket in $\SHform{M}$ through the following commutative diagram,
\[\begin{tikzcd}
	{\Hform{l}{M} \times \Hform{m}{M}} & {\Hform{1 + l + m - k}{M}}\\
	{\Hmultivector{k - l}{M} \times \Hmultivector{k - m}{M}} & {\Hmultivector{2k - l - m - 1}{M}}
	\arrow["{\{\cdot, \cdot\}}", from=1-1, to=1-2]
	\arrow["{[\cdot, \cdot]}", from=2-1, to=2-2]
	\arrow["\flat_{k - l} \times \flat_{k - m}"{description}, tail reversed, from=1-1, to=2-1]
	\arrow["\flat_{2k - l - m -1}"{description}, tail reversed, from=1-2, to=2-2]
\end{tikzcd}.\] This bracket is given by $$\{\hat{\alpha}, \hat{\beta}\} = - \widehat{ \iota_{U \wedge V} \omega},$$ where $\iota_U \omega = d \alpha$, $\iota_V \omega = d \beta,$
and satisfies the following equalities (which follow easily from the equalities of Schouten-Nijenhuis bracket \cite{vaisman2012lectures})
\begin{itemize}
    \item[$i)$] $$\{\widehat{\alpha}, \widehat{\beta}\} = (-1)^{l_1l_2}\{\widehat{\beta}, \widehat{\alpha}\};$$
    \item[$ii)$] $$(-1)^{l_1(l_3 -1)}\{\widehat{\alpha},\{\widehat{\beta}, \widehat{\gamma}\}\} + (-1)^{l_2(l_1 -1)}\{\widehat{\beta},\{\widehat{\gamma}, \widehat{\alpha}\}\} + (-1)^{l_3(l_2 -1)}\{\widehat{\gamma},\{\widehat{\alpha}, \widehat{\beta}\}\} = 0,$$
\end{itemize}
 for $\widehat{\alpha} \in \Hform{l_1}{M}, \widehat{\beta} \in \Hform{l_2}{M}, \widehat{\gamma} \in \Hform{l_3}{M}.$ However, this bracket does not define a graded Lie algebra and we need to modify the definition slightly to get a bracket that does. First, recall that a graded Lie bracket needs to satisfy
$$\deg \{\widehat{\alpha}, \widehat{\beta}\} = \deg \widehat{\alpha} + \deg \widehat{\beta},$$ for certain notion of degree. Now, since the subspace $\Hform{k-1}{M}$ is closed under $\{\cdot, \cdot\}$, we are forced to set
$$\deg \widehat{\alpha} := 0,$$ for $\alpha \in \Hform{k-1}{M}.$ Therefore, one is tempted to define $$\deg \widehat{\alpha} := k -1 - (\text{order of } \alpha),$$ for $\widehat{\alpha} \in \SHform{M}.$ And, indeed, for 
$\widehat{\alpha} \in \Hform{l}{M}, \widehat{\beta} \in \Hform{m}{M}$, we have $$\deg\{\widehat{\alpha}, \widehat{\beta}\} = k - 1 - (1 + l + m - k)= 2k - l - m - 2 = (k - 1 - l) + (k - 1 - m) = \deg \widehat{\alpha} + \deg \widehat{\beta}.$$ We can now define $$\{\widehat{\alpha}, \widehat{\beta}\}^\bullet := (-1)^{\deg \widehat \alpha} \{ \widehat{\alpha}, \widehat{\beta}\},$$ and we have that
\begin{itemize}
    \item[$i)$] $$\{\widehat{\alpha}, \widehat{\beta}\}^\bullet = -(-1)^{\deg \widehat{\alpha} \deg \widehat{\beta}} \{\widehat{\beta}, \widehat{\alpha}\}^\bullet,$$
    \item[$ii)$] $$(-1)^{\deg \widehat{\alpha} \deg \widehat{\gamma}}\{\widehat{\alpha}, \{\widehat{\beta}, \widehat{\gamma}\}^\bullet\}^\bullet + \text{cycl.} = 0.$$
\end{itemize}

Summarizing, we have proved
\begin{theorem}[\cite{HamiltonianStructuresIbort}] $(\SHform{M}, \{ \cdot, \cdot\}^\bullet)$ is a graded Lie algebra.
\end{theorem}

\begin{remark} Of course, restricting this structure to the forms of order $k-1$ we obtain the Lie algebra $(\Hform{k-1}{M}, \{ \cdot, \cdot\}^\bullet)$. This Lie algebra is of particular importance in the study of multisymplectic manifolds, since $(k-1)$-forms represent the \textit{conserved quantities} and \textit{currents} of classical field theory and calculus of variations.
\end{remark}

\begin{remark} If $(M, \omega) = (\bigwedge^k_2 L, \Omega_L),$ we can obtain a graded Lie bracket without quotienting by closed forms by restricting the bracket to the subspace of semi-basic forms. For further details, we refer to \cite{KANATCHIKOV1997225}.
\end{remark}

Similar to the characterization of coisotropic submanifold of a symplectc manifold in terms of the Poisson algebra, we can prove the following result.\\

\begin{proposition}\label{prop:subalgebracoisotropic} Let $i: N \hookrightarrow M$ be a $k$-coisotropic submanifold. Then $$ \widehat{I}_N := \{ \widehat{\alpha} \in \Hform{k-1}{M} : \,\, i^\ast d\alpha = 0\}\footnote{\text{That is, the space of all Hamiltonian $(k-1)$-forms that have a representative which is closed on $N$.}}$$ defines a subalgebra of $(\Hform{k-1}{M}, \{\cdot, \cdot\}^{\bullet}).$
\end{proposition}
\begin{proof}
Let $\widehat{\alpha}, \widehat{\beta} \in \widehat{I}_N.$ Then, there are vector fields $X_\alpha, X_\beta$ satisfying $$\iota_{X_\alpha} \omega = d \alpha, \,\, \iota_{X_\beta} \omega = d \beta.$$ Since $i^\ast d\alpha, i^\ast d\beta = 0,$ we conclude that $X_\alpha, X_\beta$ take values in $(TN)^{\perp, k} \subseteq TN + \ker \flat_1.$ Without loss of generality, we can assume that $X_\alpha$, $X_\beta$ take values in $TN$. Now, since $$\{\widehat{\alpha}, \widehat{\beta}\}^\bullet =  (-1)^{(k-1)} \widehat{ \iota_{X_\alpha \wedge X_\beta} \omega},$$ and $X_\alpha$, $X_\beta$ take values in $(TN)^{\perp, k}$ and $TN$ , we have $$  i^\ast \left(\iota_{X_\alpha \wedge X_\beta} \omega\right) = 0,$$ concluding that $$\{\widehat{\alpha}, \widehat{\beta}\}^\bullet \in \widehat{I}_N.$$
\end{proof}
\begin{remark} When $(M, \omega)$ is non-degenerate, each Hamiltonian $(k-1)$-form $\alpha$ defines an \textit{unique} vector field $X_\alpha$ satisfying $$\iota_{X_\alpha} \omega = d \alpha.$$ Therefore, the bracket $$\{\alpha, \beta\} = \iota_{X_\alpha\wedge X_\beta}\omega$$ is well defined. This, however, does not define a Lie algebra since the Jacobi identity holds up to an exact form. Nevertheless, it does defines an algebraic structure called an $L_\infty$-algebra (see \cite{Linfintyalgebra}). \cref{prop:subalgebracoisotropic} is also true in this context, that is, to each coisotropic submanifold $N$, there is the corresponding $L_\infty$-algebra of forms that are closed on $N$.
\end{remark}

Let us now briefly discuss conserved quantities. Consider a locally decomposable Hamiltonian multivector field of order $q$, $$\iota_{X_H} \omega = dH,$$ where $H \in \Omega^{k-q}(M)$ is the Hamiltonian. We will consider as a solution any immersion $\phi: \Sigma \rightarrow M,$ where $\dim \Sigma = q,$ satisfying $$\phi_\ast U = X_H,$$ where $U$ is some nowhere vanishing multivector field of order $q$ on $\Sigma.$ Then, a conserved quantity (for the solution $\phi$) is a $(q-1)$-form satisfying $$d \phi^\ast \alpha = 0.$$ In terms of possibly non-decomposable (nor integrable) multivector fields, this notion extends as follows
\begin{Def}[Conserved quantity] A conserved quantity for a Hamiltonian multivector field $X_H \in \mathfrak{X}^q(M)$ is a $(q-1)-$form $\alpha$ on $M$ satisfying $$\langle d \alpha, X_H \rangle = 0.$$ 
\end{Def}
Then, for Hamiltonian forms, we have the following
\begin{proposition} Let $X_H$ be a Hamiltonian multivector field of order $q$, with Hamiltonian form $H \in \Omega^{k-q}(M)$ and $\alpha$ be a Hamiltonian form of order $q-1$. Then $\alpha$ is a conserved quantity for $X_H$ if and only if $$\{\widehat\alpha, \widehat H\}^\bullet = 0.$$
\end{proposition}

For a treatment of conserved quantities and moment maps using the $L_\infty-$algebra strcuture of observables, we refer to \cite{ExistenceOfComoments,ConservedQuantitiesMarco}.
\subsection{Hamiltonian multivector fields as Lagrangian submanifolds}\label{Subsection:HamiltonianAsLagrangian}
Given a symplectic manifold $(M, \omega)$, we can endow its tangent bundle with a symplectic structure using the bundle ismorphism $$TM \xrightarrow{\flat} T^\ast M,$$ and the canonical symplectic form on $T^\ast M.$ With this definition and interpreting a vector field $X: M \rightarrow TM$ as a submanifold, $X$ is $1$-Lagrangian if and only if it is locally Hamiltonian. We would like to generalize this result to general multisymplectic manifolds and multivector fields of aribitrary order $q$ $$U:M\rightarrow \bigvee_qM.$$ In \cite{HamiltonianStructuresIbort}, the authors prove a generalization of the result to vector fields in multisymplectic manifolds $$X: M \rightarrow TM,$$ endowing the tangent bundle $TM$ with a multisymplectic structure via the complete lift of forms. We will explore how to generalize this method in $\cref{Subsection:Completelift}$. In the meantime, let us begin by defining a multisymplectic structure on $\bigvee_q M.$\\

Given a multisymplectic manifold $(M, \omega)$ of order $k$, we have the induced map by contraction $$\bigvee_q M \xrightarrow{\flat_q} \bigwedge^{k+1-q} M;\, u \mapsto \iota_u \omega.$$ Using the canonical multisymplectic form $\Omega_M^{k + 1 - q}$ on $\bigwedge^{k + 1 - q} M$, we can define the closed form (in fact, exact)$$\widetilde{\Omega}^q_M := (\flat_q)^\ast \Omega_M^{k +1 - q},$$ which endows $\bigvee_q M$ with a multisymplectic structure of order $(k + 1 - q)$. Notice that, for $q = 1$, the order of the multisymplectic structure on $TM$ is the order of the multisymplectic structure on $M$.

\begin{remark} Even if $\omega$ is non-degenerate, $\widetilde\Omega_M^q$ could have non trivial kernel. This motivates the study of ``general" multisymplectic structures that we have adopted along this paper, which provides a way of interpreting multivector fields as Lagrangian submanifolds of (possible degenerate)  multisymplectic manifolds.
\end{remark}

Denote by $\widetilde{W}_M^q$ the vertical distribution associated to the vector bundle $$\bigvee_q M \rightarrow M.$$ Since $\flat_q$ is a bundle map
\[\begin{tikzcd}
	{\bigvee_qM} && {\bigwedge^{k+1-q}M} \\
	& M
	\arrow["{\flat_q}", from=1-1, to=1-3]
	\arrow[curve={height=6pt}, from=1-1, to=2-2]
	\arrow[curve={height=-6pt}, from=1-3, to=2-2]
\end{tikzcd},\] we have that $$(\flat_q)_\ast \widetilde{W}_M^q \subseteq W_M^{k+1-q},$$ where $W_M^{k+1-q}$ is the vertical distribution of the vector bundle $$\bigwedge^{k+1-q}M \rightarrow M.$$
Now, recall that $W_M^{k +1 - q}$ defines a $1$-Lagrangian distribution. Therefore, we have
\begin{proposition}\label{verticalis1isotropic} $\widetilde{W}_M^q$ defines a $1$-isotropic distribution on $(\bigvee_q M, \widetilde{\Omega}_M^q).$
\end{proposition}

Now we can prove the main result of this section.
\begin{theorem} Let $(M, \omega)$ be a multisymplectic manifold of order $k$. Then, a multivector field $$U: M \rightarrow \bigvee_q M$$ is locally Hamiltonian if and only if it defines a $(k +1 - q)$-Lagrangian submanifold in $(\bigvee_q M, \widetilde{\Omega}_M^q).$
\end{theorem}
\begin{proof} With \cref{isotropicdecomposition} in mind, since $\widetilde{W}_M^q$ is $1$-isotropic by \cref{verticalis1isotropic}, and we have the decomposition $$T\left(\bigvee_q M \right)\bigg|_{U(M)} = U_\ast(TM) \oplus \widetilde{W}_M^q\big |_{U(M)},$$ we only need to check wether $U$ defines a $(k+1 -q)$-isotropic submanifold or, equivalently, wether $$U^\ast \widetilde{\Omega}_M^{q} = 0.$$
\[\begin{tikzcd}
	{\bigvee_qM} && {\bigwedge^{k +1 -q }M} \\
	\\
	M
	\arrow["{\flat_q}", from=1-1, to=1-3]
	\arrow["U", from=3-1, to=1-1]
	\arrow["{\flat_q(U) = \iota_U \omega}"', curve={height=6pt}, from=3-1, to=1-3]
\end{tikzcd}\]
We have that 
\begin{align*}
    U^\ast \widetilde{\Omega}_M^{q} &= U^\ast \flat_q^\ast \Omega_M^{k +1 - q} = (\flat_q \circ U)^\ast \Omega_M^{k +1 - q}\\
    &= (\iota_U \omega)^\ast \Omega_M^{k +1 - q} = - d \iota_U \omega,
\end{align*}
where in the last equality we have used that $\alpha^\ast \Omega^k_Q = - d \alpha,$ for any form $\alpha: Q \rightarrow \bigwedge^k Q$. We conclude that $U$ is $k$-Lagrangian if and only if $$0 = U^\ast \widetilde{\Omega}_M^q = - d \iota_U \omega,$$ that is, if and only if $U$ is locally Hamiltonian.
\end{proof}

\subsection{Complete lift of forms to multivector bundles}\label{Subsection:Completelift}
In \cite{Ibort1999OnTG}, the authors prove that $(TM, \omega^c)$ is a non-degenerate multisymplectic manifold when $\omega$ is a non-degenerate multisymplectic form on $M$. Here $\omega^c$ denotes the complete lift of the form. We would like to generalize this procedure to arbitrary multivector bundles $$\bigvee_q M.$$ Let us begin by recalling that $\omega^c$ is the unique $(k+1)$-form on $TM$ satisfying $$X^\ast \omega^c = \pounds_X \omega,$$ for every vector field $$X: M \rightarrow TM.$$ Recalling the Cartan formula $$\pounds_X \omega = d \iota_X \omega + \iota_X d \omega,$$ we define the Lie derivative of a $\omega$ with respect to a multivector field $$U: M \rightarrow \bigvee_q M$$ as the $(k + 2 - q)$-form (see \cite{TulzcyjewLieDerivative}) $$\pounds_U \omega := \iota_U d \omega + (-1)^{q+1} d \iota_U \omega.$$

\begin{theorem}[Definition of complete lift]\label{thm:completelift} Given a manifold $M$, and $\omega \in \Omega^{k +1}(M)$, there exists an unique $(k + 2 - q)$-form on $\bigvee_q M,$ $\omega^c_q$, such that $$U ^\ast \omega^c_q = \pounds_U \omega,$$ for every multivector field $$U: M \rightarrow \bigvee_q M.$$
\end{theorem}
To prove uniqueness, it suffies to study the linear problem. 

\begin{lemma}\label{lemma:completelift} Let $X, Y$ be vector spaces and $\pi: Y \rightarrow X$ be an epimorphism. Then, if $k +1 \leq \dim X$, a form $\omega \in \bigwedge^{k +1} Y^\ast$ is characterized by the pull-backs of all sections $$\phi: X \rightarrow Y.$$ That is, if there is another $(k +1)$-form $\alpha$ on $Y$ such that $\phi^\ast \alpha = \phi^\ast \omega$, for every section $\phi: X \rightarrow Y$ of $\pi,$ then $$\alpha = \omega.$$
\end{lemma}
\begin{proof} It is clear that $\omega$ is characterized by the induced linear map $$\omega: \bigwedge^{k+1} Y \rightarrow \mathbb{R},$$ and that, if $\phi^\ast \alpha = \phi^\ast \omega$, for certain form $\alpha \in \bigwedge^{k+1} Y^\ast$, the following diagram commutes.
\[\begin{tikzcd}
	{\bigwedge^{k+1}Y} & {\mathbb{R}} \\
	{\bigwedge^{k+1}X}
	\arrow["\omega", from=1-1, to=1-2]
	\arrow["{\pi_\ast}"', from=1-1, to=2-1]
	\arrow["{\phi_\ast}", curve={height=-18pt}, dashed, from=2-1, to=1-1]
	\arrow["{\phi^\ast\alpha}"', from=2-1, to=1-2]
\end{tikzcd}.\]
Therefore, if we can prove that $$\bigwedge^{k+1} Y = \left \langle \phi_\ast\left( \bigwedge^{k+1} X\right), \,\, \phi: X \rightarrow Y \text{ section }  \right\rangle,$$ we would have $\omega = \alpha,$ since they would coincide in a set of generators. Identify $X$ as a subspace of $Y$. We have
\begin{align*}
    \bigwedge^{k+1} Y = \bigwedge^{k+1}(X \oplus \ker \pi) = \bigoplus_{l = 0}^{k +1} \left(\bigwedge^{l} X \wedge \bigwedge^{k +1 - l} \ker \pi\right).
\end{align*}
We will prove that $$\bigwedge^{l} X \wedge \bigwedge^{k + 1 -l} \ker \pi \subseteq \left \langle \phi_\ast\left( \bigwedge^{k+1} X\right), \,\, \phi: X \rightarrow Y \text{ section }  \right\rangle.$$ Let $$x_1 \wedge \cdots \wedge x_{l} \wedge y_{l+1} \wedge \cdots \wedge y_{k+1} \in \bigwedge^{l} X \wedge \bigwedge^{k +1 -l} \ker \pi,$$ where $x_i \in X$, $y_j \in \ker \pi$ are linearly independent vectors. Extend $x_1, \dots, x_{k +1 - l}$ to $k +1$ linearly independent vectors on $X$ (here we are using $\dim X \geq k+1$), $$x_1, \dots, x_{k+1}$$ and  take a section $\phi: X \rightarrow Y$ such that $$\phi(x_i) = x_i, i = 1, \dots, k, \,\, \phi(x_{k+1}) = x_{k +1} + y_{k+1}.$$ Then 
\begin{align*}
    &x_1 \wedge \cdots \wedge x_{k} \wedge y_{k+1} = \\
    &\phi_\ast (x_{1} \wedge \cdots \wedge x_{k+1}) - x_{1} \wedge \cdots \wedge x_{k+1} \in \left \langle \phi_\ast\left( \bigwedge^{k+1} X\right), \,\, \phi: X \rightarrow Y \text{ section }  \right\rangle.
\end{align*}
With a similar argument we can show that $$x_1\wedge\cdots \wedge x_{k -1} \wedge y_k \wedge x_{k +1} \in \left \langle \phi_\ast\left( \bigwedge^{k+1} X\right), \,\, \phi: X \rightarrow Y \text{ section }  \right\rangle.$$ Now, defining another section (which we name the same making abuse of notation) $\phi$ satisfying $$\phi(x_i) = x_i, \, i= 1, \dots, k -1,\,\, \phi(x_{k}) = x_k + y_k, \phi(x_{k+1}) = x_{k+1} + y_{k+1},$$ we have
\begin{align*}
    \phi_\ast(x_1 \wedge \cdots \wedge x_{k+1}) = x_1 \wedge \cdots \wedge x_{k-1} \wedge (x_k + y_k) \wedge (x_{k+1} + y_{k+1})
\end{align*}
which, by the previous considerations implies 
$$x_1 \wedge \cdots \wedge x_{k-1} \wedge y_{k} \wedge y_{k+1}\in \left \langle \phi_\ast\left( \bigwedge^{k+1} X\right), \,\, \phi: X \rightarrow Y \text{ section }  \right\rangle.$$ Now, iterating this argument we conclude $$x_1 \wedge \cdots \wedge x_{l} \wedge y_{l+1} \wedge \cdots \wedge y_{k+1} \in \left \langle \phi_\ast\left( \bigwedge^{k+1} X\right), \,\, \phi: X \rightarrow Y \text{ section }  \right\rangle,$$ proving the result.
\end{proof}
\begin{proof}[Proof of \cref{thm:completelift}] By \cref{lemma:completelift}, if we find a form $\omega^c_q$ on $\bigvee_q M$ satisfying $$U^\ast \omega^c_q = \pounds_U \omega,$$ the result would follow. Consider the induced maps by $\omega$ and $d \omega$ on $\bigvee_q M$,
\[\begin{tikzcd}
	&& {\bigwedge^{k+2 - q}M} \\
	{\bigvee_q M} && {\bigwedge^{k+1 - q}M}
	\arrow["{\widetilde \flat_q:= \iota_{\bullet} d \omega}", curve={height=-6pt}, from=2-1, to=1-3]
	\arrow["{\flat_q: = \iota_{\bullet} \omega}"', from=2-1, to=2-3]
\end{tikzcd},\]
and define a $(k +2 - q)$-form on $\bigvee_q M$ by $$\omega^c_q := (\widetilde{\flat}_q)^\ast \Theta_M^{k + 2 - q} + (-1)^q (\flat_q)^\ast \Omega_M^{k + 1 - q}.$$ Then, by definition of $\Theta_M^{k +2 -q},$ and $\Omega_M^{k +1 - q}$ we have that for all multivector fields $U: M \rightarrow \bigvee_q M,$
\begin{align*}
    U^\ast \omega^c_q = \iota_U d \omega + (-1)^{q+ 1} d \iota_U \omega  = \pounds_U \omega,
\end{align*}
finishing the proof.
\end{proof}
\begin{remark}
    Now that we have generalized the complete lift of forms to arbitrary multivector bundles, given a multisymplectic manifold $(M, \omega)$ we have two ways of inducing a multisymplectic structure on $\bigvee_q M,$ the one constructed in \cref{Subsection:HamiltonianAsLagrangian}, and the complete lift from \cref{thm:completelift}. However, because $\omega$ is closed, the map $\widetilde{\flat}_q$ of the proof of \cref{thm:completelift} is trivial and thus, $$\omega^c_q = (-1)^q(\flat_q)^\ast \Omega_M^{k +1 - q} = (-1)^q \widetilde{\Omega}_M^q$$ and we conclude that, up to sign, both multisymplectic structures are equal.
\end{remark}

\section{Coisotropic submanifolds}\label{section:CoisotropicSubmanifolds}
\subsection{Local form of coisotropic submanifolds}\label{LocalFormSection}
Weinstein gave the first normal form\footnote{Along this paper, we reserve the term \textit{normal form} for a classification of a neighborhood of an entire submanifold (like in \cref{NeighborhoodTheorem}, \cref{LagrangianFormMultisymplectic}), and we use the term \textit{local form} for a classification of a neighborhood around any point of a submanifold (like in \cref{localformcoisotropic}).} theorem for Lagrangian submanifolds in the context of symplectic geometry.

\begin{theorem}[\cite{WeinsteinLagrangianNeighborhood} Weinstein's Lagrangian neighborhood Theorem]\label{NeighborhoodTheorem} Let $(M, \omega)$ be a symplectic manifold and $L \hookrightarrow M$ be a Lagrangian submanifold. Then there are neighborhoods $U$, $V$ of $L$ in $M$, and in $T^\ast L$ (identifying $L$ with the zero section) respectively, and a symplectomorphism $$\phi: U \rightarrow V.$$
\end{theorem}

This result has been generalized to multisymplectic manifolds of type $(k,0)$ by G. Martin \cite{Martin1988ADT}, and extended to multiysmplectic manifolds of type $(k,r)$ by M. de Leon et al. \cite{deleon2003tulczyjews}.

\begin{theorem}[\cite{deleon2003tulczyjews}]\label{LagrangianFormMultisymplectic}\label{normalformlagrangian} Let $(M, \omega, W, \mathcal{E})$ be a multisymplectic manifold of type $(k,r)$, and $L \hookrightarrow M$ be a $k$-Lagrangian submanifold complementary to $W$, that is, such that $$TL \oplus W \big |_L = TM \big |_L.$$ Then there are neighborhoods $U$, $V$ of $L$ in $M$, and of $L$ in $\bigwedge^k_r L$ (identifying $L$ as the zero section), where the horizontal forms are taken with respect to $\mathcal{E}$ under the identification $$TL = TM/ W,$$ and a multisymplectomorphism $$\psi: U \rightarrow V,$$ which is the identity on $L$ and satisfies $$\psi_\ast W = W^k_L,$$ where $W_L^k$ denotes the vertical distribution on $\bigwedge^k_r L^\ast.$
\end{theorem}

\begin{proof} Define the vector bundle isomorphism $$\phi: W|_L \rightarrow \bigwedge^k_r L; \,\, \phi(w_l) := (\iota_{w_l} \omega) |_L.$$ By the tubular neighborhood theorem, we may identify a neighborhood $U$ of $L$ in $W|_L$ with a neighborhood of $L$ in $M$. Under the previous identificaction, let $V := \phi(U)$ and define $$\widetilde{\omega} := \phi_\ast \omega.$$ Following the same line of reasoning as in \cref{prop: characterization}, we have $\widetilde{\omega} = \Omega^k_L$ on $L$. Furthermore, since $\phi$ is a vector bundle isomorphism, $\phi$ preserves fibers and we have $\phi_\ast W|_U = (W_L^k)|_V.$ This implies that $W_L^k$ not only defines a $1$-isotropic distribution for $\Omega^k_L$, but also for $\widetilde{\omega}.$ To build the multisymplectomorphism $\psi$, we will make use of \textit{Moser's trick} with the family of forms $$\Omega_t:= (1-t) \Omega_L^k + t \widetilde \omega.$$ More precisely, we will look for a time dependent vector field $X_t$ on $V$ such that its flow $\phi_t$ satisfies $$\phi_t ^\ast \Omega_t = \Omega^k_L,$$ for every $t$. To achieve this, it will be sufficient to look for a time dependent vector field $X_t$ such that $$0 = \dv{t}\left(\phi^\ast_t \Omega_t\right) = \pounds_{X_t} \Omega_t + \dv{\Omega_t}{t} = d \iota_{X_t} \Omega_t + \widetilde \omega - \Omega^k_L.$$ Now, if we denote by $\pi_t$ multiplication by $t$ in $\bigwedge^k_r L$, by reducing neighborhoods if necessary, we get a well defined map $$\pi_t : V \rightarrow V,$$ for $0 \leq t \leq 1$. By the relative Poincaré Lemma, we have $$\widetilde \omega = d \left( \int_0^1 \pi_t^\ast \iota_\Delta \widetilde{\omega} dt\right),$$ where $\Delta$ is the dilation vector field. Therefore, if we define $$\widetilde \theta := - \int_0^1 \pi_t^\ast \iota_{\Delta} \widetilde{\omega} dt,$$ it follows that $\widetilde \omega = - d \widetilde \theta,$ where $\widetilde \theta = 0$ on $L$ (because $\Delta = 0$ on $L$). Since we need $ \Omega ^k_L - \widetilde \omega = - d (\Theta^k_L - \widetilde \theta) = d\iota_{X_t} \Omega_t,$ it will be enough to look for $X_t$ satisfying $$\iota_{X_t} \Omega_t =  \widetilde \theta - \Theta^k_L.$$ Recall that $\widetilde \omega = \Omega^k_L$ on $L$ and, therefore $\Omega_t = \Omega^k_L$ on $L$. Since this form is nondegenerate, by reducing the neighborhoods further, we can assume that $\Omega_t$ is nondegenerate on $V$, for every $t \in [0, 1]$. Notice that $\iota_Y \left( \widetilde \theta - \Theta^k_L \right) = 0$, for any vector field $Y$ that takes values in $W^k_L,$ and that $$\iota_{E_1 \wedge \cdots \wedge E_r} \widetilde \theta = \iota_{E_1 \wedge \cdots \wedge E_r} \Theta^k_L = 0,$$ for vector fields $E_1, \dots, E_r$ such that $\pi(E_i)$ takes values in $\mathcal{E} \subset L$ (where $\pi: \bigwedge  ^k_r L \rightarrow L$ is the canonical projection). These last two properties, together with \cref{lemma1}, imply that there exists an unique time-dependent vector field $X_t$ with values in $W_L^k$ satisfying $$\iota_{X_t} \Omega_t =  \widetilde \theta - \Theta^k_L.$$ Furthermore, since $\widetilde \theta = \Theta ^k_L = 0$ on $L$, $X_t = 0$ on $L$, and its flow is globally defined on $L$. It follows that we can assume that $\phi_t$ (the flow of $X_t$) is defined on $V$ for $0 \leq t \leq 1$ by reducing the neighborhoods further. Finally, for $t = 1$, this flow satisfies $$\phi_1 ^\ast \widetilde \omega = \Omega$$ and preserves fibers, because $X_t$ takes values in $W_L^k.$ Defining $$\psi := (\phi_1)^{-1} \circ \phi,$$ we get the desired multisymplectomorphism.
\end{proof}

We can use \cref{LagrangianFormMultisymplectic} to give a local form for vertical $k$-coisotropic submanifolds $N \hookrightarrow M$ of a multisymplectic manifold of type $(k,r)$, where vertical means that $$W \big |_N \subseteq TN.$$ 

\begin{theorem}[Local form of $k$-coisotropic submanifolds relative to Lagrangian submanifolds]\label{localformCoisotropicLagrangian}Let $(M, \omega, W, \mathcal{E})$ be a multisymplectic manifold of type $(k, r)$, $i: N \hookrightarrow M$ be a $k$-coisotropic submanifold satisfying $$W|_N \subseteq TN,$$ and $L \hookrightarrow M$ be a $k$-Lagrangian submanifold complementary to $W$, that is, such that $$W |_L \oplus TL = TM |_L.$$ Then there exists a neighborhood $U$ of $L$ in $M$, a submanifold $Q \hookrightarrow L$, a neighborhood $V$ of $L$ in $\bigwedge^k_r L$, and a multisymplectomorphism $$\phi: U \rightarrow V$$ satisfying
\begin{itemize}
    \item[$i)$] $\phi$ is the identity on $L$, identified as the zero section in $\bigwedge^k_r L$;
    \item[$ii)$] $\phi(N \cap U) = \bigwedge^k_r L \big |_Q \cap V.$
\end{itemize}
\end{theorem}
\begin{proof} Let $U$, $V$, and $\phi$ be the neighborhoods and multisymplectomorphism from \cref{normalformlagrangian} and define $$Q := L \cap N.$$ We claim that, for $U,V$ small enough, $$\phi(N \cap U) = \bigwedge^k_r L \big |_Q \cap V.$$ First recall that we have $$\phi_\ast W = W_L,$$ where $W_L$ is the canonical $1$-Lagrangian distribution on $\bigwedge^k_rL.$ Let $x \in L \cap N$ and $F_x$ be the leaf of $W$ through $x$. It is clear that $F_x \subseteq N$, and that, reducing $U$ and $V$ if necessary, $$\phi(F_x \cap U) =\bigwedge^k_r T^\ast_x L \cap V,$$ since diffeomorphisms that preserve distributions preseve their leaves (when the distributions are integrable). Again, reducing $U$ and $V$ further, we may also assume that for every point $y \in N \cap U$ there is a point $x \in L \cap N$ such that the leaf of $W$ that contains $x$, $F_x$, also contains $y$, that is, we may assume that $$N \cap U = \bigcup_{x \in L \cap N} F_x \cap U.$$ Therefore, $$\phi(N \cap U) = \bigcup_{x \in L \cap N} \phi(F_x \cap U) = \bigcup_{x \in Q} \bigwedge^k_r T^\ast_x L \cap V = \bigwedge^k_rL \big |_Q \cap V,$$ proving the result.
\end{proof} 

\begin{theorem}\label{localformcoisotropic} Let $(M, \omega, W, \mathcal{E})$ be a multisymplectic manifold of type $(k,r)$, and $N \hookrightarrow M$ be a $k$-coisotropic submanifold satisfying $$W|_N \subseteq TN.$$ Then, given any point $x \in N$, there exists a neighborhood $U$ of $x$ in $M$, a manifold $L$, a submanifold $Q \hookrightarrow L$, a neighborhood $V$ of $L$ in $\bigwedge^k_r L$ and a multisymplectomorphism $$\phi: U \rightarrow V$$ such that
\begin{itemize}
    \item[$i)$] $\phi$ is the identity on $L$, idetified as the zero section in $\bigwedge^k_r L$;
    \item[$ii)$] $\phi(N \cap U) = \bigwedge^k_r L \big |_Q \cap V.$
\end{itemize}
\end{theorem}
\begin{proof} Using \cref{thm:darbouxmultisymplectic}, we can build a $k$-Lagrangian submanifold $L$ through any given point $x \in N.$ Now the result follows using \cref{localformCoisotropicLagrangian}.
\end{proof}
\subsection{Coisotropic reduction}
When $k = 1$, that is, when $(M, \omega)$ is a symplectic manifold, we have the classical result of coisotropic reduction due to Weinstein \cite{weinstein1977lectures}.

\begin{theorem}\label{CoisotropicReductionSymplectic}[Coisotropic reduction in symplectic geometry] Let $(M, \omega)$ be a symplectic manifold, $i: N \hookrightarrow M$ be a coisotropic submanifold, and $j: L \hookrightarrow M$ be a Lagrangian submanifold that has clean intersection with $N$.  Then, $TN^{\perp}$ is an integrable distribution and determines a foliation $\mathcal{F}$ of maximal integral leaves. Suppose that the quotient space $N/ \mathcal{F}$ admits an smooth manifold structure such that the canonical projection $$\pi: N \rightarrow N/\mathcal{F}$$ defines a submersion. Then there exists an unique symplectic form on $N/\mathcal{F}$, $\omega_N$ compatible with $\omega$ in the following sense $$\pi^\ast \omega_N = i^\ast \omega.$$ Furthermore, if $\pi(N \cap L)$ is a submanifold, it is Lagrangian in $(N/\mathcal{F}, \omega_N).$
\end{theorem}

We would like to find an analogous result in multisymplectic manifolds. For the first part of \cref{CoisotropicReductionSymplectic}, the classical argument works.

\begin{proposition}[\cite{Ibort1999OnTG}]\label{Prop:Involutivity} Let $(M, \omega)$ be a multisymplectic manifold of order $k$ and $i:N\hookrightarrow M$ be a \kcoiso submanifold. Then, $(TN)^{\perp, k} \cap TN \subseteq TN$ defines an involutive distribution.
\end{proposition}
\begin{proof} Let $X, Y \in \mathfrak{X}(N)$ be vector fields on $N$ with values in $(TN)^{\perp, k},$ and let $Z_1, \dots, Z_k \in \mathfrak{X}(N)$ be arbitrary vector fields on $N$. Denote $$\omega_0 := i^\ast \omega.$$ Since $\omega$ is closed, we have
\begin{align*}
    0 &= (d \omega_0)(X, Y, Z_1, \dots, Z_k) =  X(\omega_0(Y, Z_1, \dots, Z_k)) - Y(\omega_0(X, Z_1, \dots, Z_k))\\
    &+ \sum_{j = 1}^k (-1)^j Z_i(\omega_0(X, Y, Z_1 \dots, \hat{Z}_i, \dots, Z_k)) - \omega_0([X, Y], Z_1, \dots, Z_k)\\
    &+ \sum_{j = 1}^k (-1)^{j +1}\omega_0([X, Z_i], Y, Z_1 \dots, \hat{Z}_i, \dots, Z_k)\\
    &+ \sum_{j = 1}^k (-1)^{j}\omega_0([Y, Z_i], X, Z_1 \dots, \hat{Z}_i, \dots, Z_k)\\
    &+ \sum_{i < j} (-1)^{i +j} \omega_0([Z_i, Z_j], X, Y, Z_1 \dots, \hat{Z}_i, \dots,\hat{Z}_j, \dots, Z_k).
\end{align*}
Now, since both $X$ and $Y$ take values in $(TN)^{\perp, k}$, all the summands but $$\omega_0([X, Y], Z_1, \dots, Z_k)$$ are zero. Therefore, we conclude $$\omega_0([X, Y], Z_1, \dots, Z_k) = 0,$$ for all $Z_1, \dots, Z_k \in \mathfrak{X}(N),$ that is, $[X, Y]$ takes values in $(TN)^{\perp, k}$, proving that the distribution is involutive.
\end{proof}

If $(TN)^{\perp, k} \cap TN$ is regular, by Frobenius' Theorem, it determines a foliation $\mathcal{F}$ of maximal leaves. We have the following result.

\begin{theorem}[\cite{Ibort1999OnTG}]\label{CoisotropicReductionMultisymplectic} Let $(M, \omega)$ be a multisymplectic manifold of order $k$, and $i: N \hookrightarrow M$ be a \kcoiso submanifold such that $(TN)^{\perp, k} \cap TN$ is regular. Suppose that $N/\mathcal{F}$ admits a smooth manifold structure such that the canonical projection $$\pi: N \rightarrow N/\mathcal{F}$$ defines a submersion. Then there exists an unique multisymplectic form of order $k$ on $N/\mathcal{F}$, $\omega_{N}$, that is compatible with $\omega$, that is, $$\pi^\ast \omega_N = i ^\ast \omega.$$
\end{theorem}
\begin{proof} Let $x \in N$. Notice that, since $\pi$ defines a submersion, we have the identification $$T_{[x]} N/\mathcal{F} = T_x N /\ker d_x \pi = T_xN/ (T_xN)^{\perp, k} \cap T_xN.$$ Let $v_1, \dots, v_{k+1} \in T_x N.$ The relation $\pi^\ast \omega_N = i^\ast \omega$ forces us to define $$\omega_N |_{[x]}([v_1], \dots, [v_{k+1}]) := \omega|_x(v_1, \dots, v_{k+1}),$$ proving that $\omega_N$ is unique. It only remains to show that the previous definition does not depend on the choice of $x$ and $v_i$. For the latter, first observe that if $[v] = 0$, that is, $v \in (T_xN)^{\perp, k}$ we have $$\omega(v, v_1, \dots, v_k) = 0,$$ for all $v_1, \dots, v_k \in T_xN.$ Therefore, if $[v_i] = [u_i],$ for $i = 1, \dots, k+1$, we have
\begin{align*}
    \omega_N|_{[x]}([v_1], \dots, [v_{k+1}]) &= \omega|_x (v_1, \dots, v_{k+1}) = \omega|_x(u_1, v_2, \dots, v_{k+1}) = \dots \\
    &= \omega|_x(u_1, \dots, u_{k+1}) = \omega_N|_{[x]}([u_1], \dots, [u_{k+1}]).
\end{align*}
For the independence of the chosen point, given $x, y \in N$ in the same leaf, we can find a complete vector field $X$ on $N$ with values in $(TN)^{\perp, k}$ such that its flow satisfies $$\phi^X_1(x) = y.$$ Now, denoting $\omega_0 := i^\ast \omega$, we have $$\pounds_X \omega_0 = \iota_X d \omega_0 + d \iota_X \omega_0 = 0,$$ since $\omega_0$ is closed and $\iota_X \omega_0 = 0$ (given that $X$ takes values in $(TN)^{\perp, k}$). This implies $(\phi^X_1) ^\ast \omega_0 = \omega_0.$ In particular, given $v_1, \dots, v_{k+1} \in T_x N$ we have
\begin{align*}
    \omega_N |_{[x]}([v_1], \dots, [v_{k+1}]) &= \omega_0 |_x(v_1, \dots, v_{k+1}) = \omega_0 |_y(d_x \phi^X_1 \cdot v_1, \dots, d_x\phi^X_1 \cdot v_{k+1})\\
    &= \omega_N|_{[y]}([d_x \phi^X_1 \cdot v_1], \dots, [d_x\phi^X_1 \cdot v_{k+1}]).
\end{align*}
Since $X$ is tangent to $\mathcal{F}$, its flow $\phi^X_1$ leaves invariant the foliation, and $\pi \circ \phi = \pi$. In particular, $$[v_i] = d_x \pi \cdot v_i = d_y \pi \cdot d_x \phi \cdot v_i = [d_x \phi \cdot v_i].$$ Finally, if $v_1, \dots, v_{k+1} \in T_x N$, $u_1, \dots, u_{k+1} \in T_yN$ with $[v_i] = [u_i]$,
\begin{align*}
    \omega_N|_{[x]}([v_1], \dots, [v_{k+1}]) &= \omega_0|_x(v_1, \dots, v_{k+1}) = \omega_0|_y(d_x \phi^X_1 \cdot v_1, \dots, d_x\phi^X_1 \cdot v_{k+1})\\
    &= \omega_N|_{[y]}([d_x \phi^X_1 \cdot v_1], \dots, [d_x\phi^X_1 \cdot v_{k+1}])\\
    &= \omega_N|_{[y]}([u_1], \dots, [u_{k+1}]),
\end{align*}
proving the result.
\end{proof}

For the projection of Lagrangian submanifolds, the second part of \cref{CoisotropicReductionSymplectic}, multisymplectic manifolds are \textit{too general} and hard to study without asking for further structures. Indeed, we can easily find a counterexample.
\begin{example}[A counterexample] Let $L = \langle l_1, l_2, l_3\rangle$ be a $3$-dimensional vector space and define $$V:= L \oplus \bigwedge^2 V^\ast.$$ Let $l^1, l^2, l^3$ be the dual basis induced on $L^\ast$ and denote $$\alpha^{ij} := l^i \wedge l^j.$$ Then $$V = \langle l_1, l_2, l_3, \alpha^{12}, \alpha^{13}, \alpha^{23}\rangle.$$ Let $l^1, l^2, l^3, \alpha_{12}, \alpha_{13}, \alpha_{23}$ be the dual basis. We have $$\Omega_L = \alpha_{12} \wedge l^1 \wedge l^2 + \alpha_{13} \wedge l^1 \wedge l^3 + \alpha_{23}\wedge l^2 \wedge l^3.$$ Define $$N := \langle l_1 + l_2, l_1 + \alpha^{23}, l_2 + \alpha^{13}, l_3, \alpha^{12}\rangle.$$ Then $N$ is a $2$-coisotropic subspace. Indeed, a quick calcultion shows $N^{\perp, 2} = 0$. This implies that the quotient space $N/N^{\perp, 2}$ is (isomorphic to) $N.$ Now, taking as the $2$-Lagrangian subspace $L = \langle l_1, l_2,l_3\rangle$, we have $$L \cap N = \langle l_1 + l_2 , l_3 \rangle.$$ However, this does not define a $2$-Lagrangian subspace of $(N, \Omega_L |_N)$, since $\alpha^{12} \in (N \cap L)^{\perp, 2}$, but $ \alpha^{12} \not \in(L \cap W).$
\end{example}
Nervertheless, we will be able to find a generalization of the previous theorem restricting the study to a particular class, those that locally are bundles of forms, which are precisely the multisymplectic manifolds appearing in classical field theories \cite{gotay2004momentum}. More particularly, we will study coisotropic reduction of vertical coisotropic submanifolds in multisymplectic manifolds of type $(k,r)$.\\

The classical proof of the last part of \cref{CoisotropicReductionSymplectic} uses en elaborate comparison of dimensions argument (see \cite{abraham2008foundations}). This argument hardly translates to multisymplectic manifolds since, in general, the map $$TM \xrightarrow{ \flat_1} \bigwedge^k M$$ does not define a bundle isomorphism. However, we can prove it using the local form proved in \cref{LocalFormSection}.\\

\label{NomalFormCoisotropicReduction} Given some manifold $L$, and a regular distribution on $L$, $\mathcal{E}$, define $$M := \bigwedge^k_r L$$ endowed with its canonical multisymplectic structure. Here, the horizontal forms, are taken with respect to $\mathcal{E}$. Let $i: Q \hookrightarrow L$ be a submanifold of dimension at least $k$ (for $\bigwedge^k Q$ to be non-zero) and take $$N:= \bigwedge^k_rL \big |_Q$$ the restricted bundle to $Q$. Then, $N \hookrightarrow M$ is a  \kcoiso submanifold. Indeed, under the (non-canonical) identification $$T_{(x, \alpha)}N = T_xQ \oplus \bigwedge^k_r T^\ast_xL,$$ for $(x, \alpha) \in N$, we have $$(TN)^{\perp, k} = 0 \oplus \ker i^\ast,$$ where $i^\ast$ is the induced map $$i^\ast: \bigwedge^k_r T^\ast_x L \subseteq \bigwedge^k T^\ast_x L \rightarrow \bigwedge^k T^\ast_x Q.$$ We claim that the image of $\bigwedge^k_r T^\ast_x L$ under $i^\ast$ is $\bigwedge ^k_r T^\ast_x Q$, where the horizontal forms are taken with respect to the subspace $$\widetilde{\mathcal{E}}_x := \mathcal{E}_x \cap T_xQ.$$ Indeed, it is clear that $$i^\ast \left(\bigwedge^k_r T^\ast_xL \right) \subseteq \bigwedge^k_r T^\ast_x Q,$$ since, if $e_1, \dots, e_r \in \widetilde{\mathcal{E}}_x$ and $\alpha \in \bigwedge^k_r T^\ast_xL$, we have $$\iota_{e_1 \wedge \dots \wedge e_r} i ^\ast \alpha = i^\ast(\iota_{e_1 \wedge \dots \wedge e_r} \alpha) = 0.$$ Now, to see the other inclusion, we take a projection $$p: T_xL \rightarrow T_xQ$$ that satisfies $p(\mathcal{E}_x) = \widetilde{\mathcal{E}}_x,$ that is, a projection that makes the following diagram commutative
\[\begin{tikzcd}
	{\mathcal{E}_x} && {T_xL} \\
	{\widetilde{\mathcal{E}}_x} && {T_xQ}
	\arrow[hook, from=1-1, to=1-3]
	\arrow[hook, from=2-1, to=2-3]
	\arrow["{p|_{\mathcal{E}_x}}", from=1-1, to=2-1]
	\arrow["p", from=1-3, to=2-3]
\end{tikzcd}.\]
Take $\beta \in \bigwedge^k_r T_x ^\ast Q$ and define $\alpha \in \bigwedge^k T^\ast_xL$ as $$\alpha:= p^\ast \beta.$$ It is clear that $i^\ast \alpha = \beta.$ Furthermore, since $p$ satisfies $p(\mathcal{E}_x) = \widetilde{\mathcal{E}}_x,$ we have $$\alpha \in \bigwedge^k_r T^\ast_xL,$$ proving that $$i^\ast \left(\bigwedge^k_r T^\ast_xL \right) = \bigwedge^k_r T^\ast_x Q.$$ In particular, when $\mathcal{E} \cap TQ$ has constant rank, so does\footnote{Because $\ran (TN)^{\perp, k} = \ran \ker i^\ast = \ran \bigwedge^k_r L - \ran \bigwedge^k_r Q$} $(TN)^{\perp, k}$, and we have that the maximal integral leaf of this distribution that contains $(x,0)$ is $$\ker i^\ast |_{\{x\}} = \{(x, \alpha): \alpha \in \ker i^\ast, \alpha \in \bigwedge^k_r T^\ast_x L\}.$$ These leaves define a vector subbundle
\[\begin{tikzcd}
	{\ker i^\ast} && {\bigwedge^k_rL\big|_Q} \\
	& Q
	\arrow[hook, from=1-1, to=1-3]
	\arrow[from=1-1, to=2-2]
	\arrow[from=1-3, to=2-2]
\end{tikzcd}.\]
By the previous considerations, these bundles fit in a short exact sequence
\[\begin{tikzcd}
	0 & {\ker i^\ast} & {\bigwedge^k_rL} & {\bigwedge^k_rQ} & 0
	\arrow[hook, from=1-2, to=1-3]
	\arrow["{i^\ast}", from=1-3, to=1-4]
	\arrow[from=1-4, to=1-5]
	\arrow[from=1-1, to=1-2]
\end{tikzcd}.\]
Therefore, we may identify $$N/\mathcal{F} = \bigwedge^k_r Q,$$ where the horizontal forms are taken with respect to $\widetilde{\mathcal{E}} = \mathcal{E}\cap TQ$ (which we are assuming to have constant rank). A routine check shows that the \ms structure induced from \cref{CoisotropicReductionMultisymplectic} is none other than the canonical \ms structure on $\bigwedge^k_rQ.$\\

Now, let us study the projection of Lagrangian submanifolds. An important class of $k$-Lagrangian submanifolds in $\bigwedge^k_r L$ are given by closed forms (\cref{prop:KLagrangianclosed}) $$\alpha: L \rightarrow \bigwedge^k_r L.$$ We have the following diagram 
\[\begin{tikzcd}
	L && {\bigwedge^k_r L = M} \\
	Q && {\bigwedge^k_rL \big|_Q =N} \\
	&& {\bigwedge^k_r Q = N/\mathcal{F}}
	\arrow["\alpha", from=1-1, to=1-3]
	\arrow[hook, from=2-1, to=1-1]
	\arrow["{\alpha |_Q}", from=2-1, to=2-3]
	\arrow[hook, from=2-3, to=1-3]
	\arrow["{i ^\ast = \pi}", from=2-3, to=3-3]
	\arrow["{i^\ast \alpha}"', curve={height=12pt}, dashed, from=2-1, to=3-3]
\end{tikzcd}.\]
It is clear that the projection of $\alpha(L) \cap N$ onto $N/\mathcal{F} = \bigwedge^k_r Q$ is exactly the image of $$i^\ast \alpha: Q \rightarrow \bigwedge^k_r Q.$$ Since $\alpha$ is closed, so is $i^\ast \alpha$, proving that in this local form, $k$-Lagrangian submanifolds complementary to the vertical distribution $\mathcal{W}$ reduce to $k$-Lagrangian submanifolds. Therefore, using \cref{localformCoisotropicLagrangian} we have the main result of this section:
\begin{theorem} Let $(M, \omega, \mathcal{W}, \mathcal{E})$ be a multisymplectic manifold of type $(k,r)$, $i:N \hookrightarrow M$ a $k$-coisotropic submanifold satisfying $$\mathcal{W} \big |_N \subseteq TN,$$ and $j:L \hookrightarrow M$ a \klagran submanifold complementary to $\mathcal{W}.$ Suppose that $N/\mathcal{F}$ admits a smooth manifold strcuture such that $\pi : N \rightarrow N/\mathcal{F}$ defines a submersion, where $\mathcal{F}$ is the foliation associated to $(TN)^{\perp, k}$ (see \cref{CoisotropicReductionMultisymplectic}), and that $$\mathcal{E} \big |_N \cap \left( TN/ \mathcal{W}\big |_N\right)$$ has constant rank. Then, if $\pi(L \cap N)$ is a submanifold, it is \klagran.
\end{theorem}

\section{Conclusions and further work}\label{conclusions}
In this paper we have analysed the role that Lagrangian and coisotropic submanifolds play in multisymplectic geometry, with the intention of extending as far as possible the well-known results in symplectic geometry. When dealing with forms of degree higher than 2, there are different complements to a submanifold, which enriches the geometry but at the same time makes it more complex. One of the first results obtained is the interpretation of Lagrangian submanifolds as possible dynamics, as well as the introduction of a graded bracket algebra. This makes it possible to deal with currents and conserved quantities. The main result of the paper is a coisotropic reduction theorem which we hope will be useful in applications to multisymplectic field theory.\\

In future work we have proposed the following objectives:

\begin{enumerate}

\item Apply the results obtained in the current paper to multisymplectic field theories.

\item Since some field theories are singular, we would like to develop a regularization method as in the case of singular Lagrangian dynamics (see \cite{ibort-marin}); previously, we have to prove a coisotropic embedding theorem \'a la Gotay \cite{gotay,zambon} in the context of multisymplectic geometry.

\item  Develop the covariant approach through a space-time decomposition, and interpret the coisotropic reduction in the corresponding infinite dimensional setting.

\item Following the notion of multi-Dirac (and higher Dirac) structures \cite{joris_MultiDirac1,jorisDirac2, ZambonDirac, BursztynDirac}, we would like to develop a more general extension using the graded Poisson brackets defined in the current paper.

\item Extend the results to the realm of multicontact geometry (see \cite{multicontacto}).

\end{enumerate}

\section*{Acknowledgements}

We acknowledge the financial support of Grant PID2022-137909NB-C2, the Severo Ochoa Programme for Centres of Excellence in R\&D (CEX2019-000904-S), and Severo Ochoa scholarship for master students. 

\phantomsection
\addcontentsline{toc}{section}{References}
\printbibliography
\end{document}